\documentclass[oneside,english]{amsart}
\usepackage[T1]{fontenc}
\usepackage[utf8]{luainputenc}
\usepackage{color}
\usepackage{babel}
\usepackage{refstyle}
\usepackage{amsthm}
\usepackage{amssymb}
\usepackage[all]{xy}
\usepackage[unicode=true,pdfusetitle,
 bookmarks=true,bookmarksnumbered=false,bookmarksopen=false,
 breaklinks=false,pdfborder={0 0 1},backref=page,colorlinks=true]
 {hyperref}

\makeatletter

\AtBeginDocument{}
\AtBeginDocument{}
\AtBeginDocument{}
\AtBeginDocument{}
\AtBeginDocument{}
\AtBeginDocument{\providecommand\eqref[1]{\ref{eq:#1}}}
\AtBeginDocument{}
\AtBeginDocument{}
\AtBeginDocument{}

\RS@ifundefined{subref}
  {\def\RSsubtxt{section~}\newref{sub}{name = \RSsubtxt}}
  {}
\RS@ifundefined{thmref}
  {\def\RSthmtxt{theorem~}\newref{thm}{name = \RSthmtxt}}
  {}
\RS@ifundefined{lemref}
  {\def\RSlemtxt{lemma~}\newref{lem}{name = \RSlemtxt}}
  {}

\numberwithin{equation}{section}
\numberwithin{figure}{section}
  \theoremstyle{plain}
  \newtheorem*{thm*}{\protect\theoremname}
\theoremstyle{plain}
\newtheorem{thm}{\protect\theoremname}[section]
  \theoremstyle{plain}
  \newtheorem{prop}[thm]{\protect\propositionname}
  \theoremstyle{definition}
  \newtheorem{defn}[thm]{\protect\definitionname}
  \theoremstyle{plain}
  \newtheorem{lem}[thm]{\protect\lemmaname}
  \theoremstyle{plain}
  \newtheorem{cor}[thm]{\protect\corollaryname}
  \theoremstyle{remark}
  \newtheorem{rem}[thm]{\protect\remarkname}
  \theoremstyle{claim}
  
  \newtheorem{example}[thm]{\protect\examplename}
\theoremstyle{note}
  
\theoremstyle{acknowledgment}
  \newtheorem{acknowledgment}[thm]{\protect\acknowledgmentname}
\theoremstyle{fact}
  
\makeatother
  
\makeatother
  \providecommand{\corollaryname}{Corollary}
  \providecommand{\definitionname}{Definition}
  \providecommand{\lemmaname}{Lemma}
  \providecommand{\claimname}{Claim}
  \providecommand{\propositionname}{Proposition}
  \providecommand{\remarkname}{Remark}
  \providecommand{\theoremname}{Theorem}
\providecommand{\theoremname}{Theorem}
\providecommand{\notename}{Note}
\providecommand{\examplename}{Example}
\providecommand{\acknowledgmentname}{Acknowledgment}
\providecommand{\factname}{Fact}
\providecommand{\sublemmaname}{Sublemma}
\begin{document}

\title{Division algebras graded by a finite group}

\author{Eli Aljadeff$^{*}$, Darrell Haile$^{**}$ \and Yaakov Karasik$^{***}$}

\address{* Department of Mathematics, Technion - Israel Institute of Technology,
Haifa 32000, Israel and Technion-Guangdong Israel Institute of Technology, Shantou, Guangdong, China.}

\address{** Department of Mathematics, Indiana University at Bloomington, Bloomington, IN 47405-7000.}

\address{*** Mathematisches Institut, Justus-Liebig-Univerist\"{a}t, Giessen, Giessen, Germany.}


\email{aljadeff 'at' technion.ac.il (E. Aljadeff),}

\email{haile 'at' indiana.edu (D. Haile).}

\email{Igor.Karasik 'at' math.uni-giessen.de (Y. Karasik),}

\keywords{graded algebras, graded division algebras, division algebra which are $G$-graded, $G$-graded simple algebras,  $k$-forms of algebras, twisted forms}

\thanks {The first author was partially supported by the ISRAEL SCIENCE FOUNDATION
(grant No. 1516/16).}
\thanks {Much of this work was done while the second author was a visitor at Guangtong-Technion Israel Institute of Technology. He would like to thank the institute for its support}

\begin{abstract}
Let $G$ be a finite group and $D$ a division algebra faithfully $G$-graded, finite dimensional over its center $K$, where $char(K) = 0$. Let $e \in G$ denote the identity element and suppose $K_{0} = K\cap D_{e}$, the $e$-center of $D$, contains $\zeta_{n_{G}}$, a primitive $n_{G}$-th root of unity, where $n_{G}$ is the \textit{exponent} of $G$. To such a $G$-grading on $D$ we associate a normal abelian subgroup $H$ of $G$, a positive integer $d$ and an element of $Hom(M(H), \mu_{n_{H}})^{G/H}$. Here $\mu_{n_{H}}$ denotes the group of $n_{H}$-th roots of unity, $n_{H} = exp(H)$, and $M(H)$ is the Schur multiplier of $H$. The action of $G/H$ on $\mu_{n_{H}}$ is trivial and the action on $M(H)$ is induced by the action of $G$ on $H$.

Our main theorem is the converse:  Given an extension $1\rightarrow H\rightarrow G\rightarrow Q\rightarrow 1$, where $H$ is abelian, a positive integer $d$, and an element of $Hom(M(H), \mu_{n_{H}})^Q$, there is a division algebra as above that realizes these data. We apply this result to classify the $G$-graded simple algebras whose $e$-center is an algebraically closed field of characteristic zero that admit a division algebra form whose $e$-center contains $\mu_{n_{G}}$.
\end{abstract}

\maketitle

\section{introduction}

If $A$ is an algebra over a field $k$ and $G$ is a finite group,  we say $A$ is {\it $G$-graded} if there is a decomposition $A=\oplus_{g\in G}A_g$ of $A$ into a direct sum of $k$-vector spaces such that $A_{g_1}A_{g_2}\subset A_{g_1g_2}$ for all $g_1,g_2\in G$.  We say the grading is {\it faithful} if $A_g\not= 0$ for all $g\in G$.  In this paper we are interested in the case where $A=D$ is a division algebra over $k$, finite dimensional over its center $K$ and faithfully graded by a finite group $G$. We assume throughout that the field $k$ is of characteristic zero and that, unless otherwise stated, $k$ contains the group $\mu_{n_{G}}$ of $n_{G}$-th roots of unity,  where $n = exp(G)$.

\begin{rem} \label{finite dimensional remark} We emphasize that we are not assuming there is a field $F$ containing $k$ such that $D$ is both $G$-graded and finite dimensional over $F$. In fact the existence of such a field will be a consequence of our results (see Theorem \ref{general structural theorem}).
\end{rem}
We begin with a brief description of our results followed by a more detailed discussion:  Writing $D=\oplus_{g\in G}D_g$ we see easily that $D_e$ is a division algebra of degree $d$ (say) over its center $L$, and $k=k\cdot 1\subseteq D_e$  (However it is not necessarily the case that the center $K$ of $D$ is contained in $D_e$ or that $K$ is even a graded subalgebra of $D$). The group $G$ acts on the center $L$ of $D_e$ and we let $H$ denote the kernel of the action.  We obtain in this way an extension $1\rightarrow H\rightarrow G\rightarrow G/H\rightarrow 1$ and a   two-cocycle $\alpha$ with class in $H^2(H, L^{*})$. It turns out (but is not trivial) that the condition that $k$ contains $\mu_{n_{G}}$ forces $H$ to be abelian. Moreover the class of $\alpha$ is $G/H$-invariant, that is, fixed by the natural action of $G/H$ on $H^2(H, L^{*})$ and gives rise to a $G/H$-invariant element in $Hom(M(H), \mu_{n_{H}})$, where $n_{H} = exp(H)$ and $M(H)$ is the Schur multiplier of $H$.

Our main result (Theorem \ref{Main realizable triple theorem}) is the converse:  Given an extension of groups $1\rightarrow H\rightarrow G\rightarrow Q\rightarrow 1$ where $H$ is abelian, a positive integer $d$, and a $Q$-invariant element of $Hom(M(H), \mu_{n_{H}})$, where the action of $Q$ on $\mu_{n_{H}}$ is trivial, there is a field $k$ containing $\mu_{n_{H}}$ and a division algebra $D$ over $k$, faithfully graded by the group $G$ and giving rise to the prescribed data. We apply our result to classify the $G$-graded simple algebras over an algebraically closed field of characteristic zero that admit a division algebra form faithfully $G$-graded over a field $k$ containing $\mu_{n_{G}}$ (see Theorem \ref{Simple algebras forms}).

\begin{rem}
As the reader may have noted two subgroups of the group of roots of unity will play an important role in the paper, namely $\mu_{n_{G}}$ where $n = exp(G)$ and $\mu_{n_{H}}$ where $n = exp(H)$. As mentioned above, a $G$-grading on a division algebra $D$ over $k$, finite dimensional over its center gives rise to a certain normal subgroup $H$ of $G$, and the existence of $\mu_{n_{G}}$ in $k$ forces the group $H$ to be abelian. Also, the $G$-grading on $D$ gives rise to an element in $Hom(M(H), L^{*})$ and it is well known (and easy to prove) that if $H$ is abelian $\exp(M(H)) \mid exp(H)$. It follows that because $\mu_{n_{H}} \leq k^{*} \leq L^{*}$, the natural inclusion $Hom(M(H), \mu_{n_{H}}) \hookrightarrow Hom(M(H), L^{*})$ is in fact onto.
\end{rem}

We proceed to the more detailed description. We begin the paper by determining the general structure of gradings by the group $G$ on $K$-central division algebras (so $K$ contains the base field $k$ which is assumed to contain $\mu_{n_{G}}$) and break the analysis into two cases:

 \begin{enumerate}

 \item
The case in which  the center $K$ is contained in the $e$-homogeneous component.

 \item

The general case, in which  the center $K$ may not be contained in $D_{e}$,  and  may not be $G$-graded.

 \end{enumerate}

There are  three types of gradings on division algebras. Our first two theorems (for the two cases mentioned above) will show that  any  grading by a finite group is a combination of these three types:

\begin{enumerate}
\item
The trivial grading: For $D$ any finite dimensional division algebra and $G=\{1\}$.
\item
The case of twisted group algebras:  Let $H$ be a finite group and consider a twisted group algebra $L^{\alpha}H$ where $\alpha$ is a $2$-cocycle $H \times H \rightarrow L^{*}$, and  where the action of $H$ on $L^{*}$ is trivial. Because $char(L) = 0$ the twisted group algebra is semisimple but in general not simple. It seems to be difficult to determine the finite groups $H$ which admit a $2$-cocycle over an aribitrary field $L$ of characteristic zero such that the twisted group algebra $L^{\alpha}H$ is a division algebra.  In our case, in which the base field contains $\mu_{n_{H}}$, we will see that the group $H$ must be abelian (Proposition \ref{abelian group twisted group algebra proposition}). It is easy to show that for any abelian group $H$   we can find a cocycle $\alpha$ and a field $L$ such that the twisted group algebra is a field and so in particular a finite dimensional division algebra. The  case where $L^{\alpha}H$ is central over $L$ is of special interest. These are obtained precisely when $H$ is abelian of the form $A \times A$ and the cocycle is a suitable nondegenerate cocycle on $H$ (see Remark \ref{central type remark} below).

\item

The crossed product grading:  If $L/K$ is a Galois extension with $G$ as Galois group, we can construct the skew group algebra $L_{t}G$.  As a left $L$-vector space ,  $L_tG$  is isomorphic to the group algebra $LG$.   We denote its elements by $\sum_{\sigma \in G}l_{\sigma}u_{\sigma}$, where $l_{\sigma} \in L$ and $u_{\sigma}$ is a symbol in $L_{t}G$ corresponding to the element $\sigma \in G$. The product is defined using distributivity and so as to satisfy the condition
$$lu_{\sigma}yu_{\tau} = l\sigma(y)u_{\sigma\tau}.$$

It is well known that the algebra $L_{t}G$ is isomorphic to $M_{n}(K)$ where $n = ord(G)$. In particular the skew group algebra structure gives a $G$-grading on $M_{n}(K)$.
Now we can twist the multiplication by a $2$-cocycle on $G$ with coefficients in $L^{*}$, where the action of $G$ on $L^{*}$ is induced by the Galois action on $L$. If $\alpha: G \times G \rightarrow L^{*}$ is such a $2$-cocycle we construct the crossed product algebra, denoted by $L_{t}^{\alpha}G$, which again is isomorphic to $LG$ and $L_{t}G$ as a left vector space over $L$ with the product  determined by the formula
$$lu_{\sigma}yu_{\tau} = l\sigma(y)\alpha(\sigma, \tau)u_{\sigma\tau}.$$

Cohomologous $2$-cocycles in $H^{2}(G, L^{*})$ yield $G$-graded isomorphic algebras over $K$ and in particular isomorphic algebras over $K$. On the other hand, noncohomologous $2$-cocycles yield nonisomorphic $K$-central simple algebras and in particular nonisomorphic $G$-graded algebras.

The crossed product $L_{t}^{\alpha}G$ is $K$-central simple and is a $G$-graded twisted form of $L_{t}G$. It is well known (we provide a proof below) that for any finite group $G$ one can find a $G$-Galois extension $L/K$ and a cocycle $\alpha$ such that $L_{t}^{\alpha}G$ is a $K$-central division algebra. Moreover one can find such crossed product division algebras for which the center contains an algebraically closed field of characteristic zero and in particular contains $\mu_{n_{G}}$. Note that here the $e$-component is $L$, a maximal subfield  in $L_{t}^{\alpha}G$.

Much more generally a finite dimensional algebra $A$ faithfully graded by the group $G$ is called a ring theoretic $G$-crossed product if every homogeneous component contains an invertible element.  The standard notation is $A_e* G$,  where $A_e$ denotes the $e$-component of the grading. Of course in this sense every division algebra faithfully graded by $G$ is a $G$-crossed product.  The (classical) crossed product $L_{t}^{\alpha}G$  described above is a special case and the source of the name.  Moreover if $H$ is a normal subgroup of $G$ then $A$ may also be viewed as a $G/H$-crossed product, where the $e$-component is now $A_H=\oplus_{h\in H}A_h$.  So the crossed product $A_e*G$ is also a crossed product  $A_H*G/H$.  This kind of construction, originally due to Teichmuller, can be found in $( $\cite{Passman}, \cite{Teichmuller} and \cite{Tignol}$)$.  We will see that in the case of a division algebra faithfully graded by a group $G$ over a field containing $\mu_{n_{G}}$, we can say much more: In particular we can choose the normal subgroup $H$ to be abelian with the division subalgebra $D_H$ a twisted group algebra over $D_{e}$ (see case (2) above).

\end{enumerate}

\begin{rem} \label{central type remark}
Classically, a finite group $G$ is said to be of central type if it has a complex irreducible representation of degree $[G:Z]^{1/2}$ where $Z=Z(G)$ is the center of $G$. In particular the index of $Z$ in $G$ must be a square. Because a linear representation of $G$ gives rise to a $2$-cocycle $\alpha$ on $G/Z$ with values in $\mathbb{C}^{*}$ and to a projective representation of $G/Z$ corresponding to $\alpha$, researchers borrowed that terminology and defined (nonclassically) a group $\Lambda$ to be of central type if $\Lambda$  admits a $2$-cocycle $\alpha \in Z^{2}(\Lambda, \mathbb{C}^{*})$ such that the twisted group algebra $\mathbb{C}^{\alpha}\Lambda \cong M_{n}(\mathbb{C})$, where $n^{2} = ord(\Lambda)$. Such a cocycle $\alpha$ is called nondegenerate. In this article we slightly extend the terminology and say that a $2$-cocycle $Z^{2}(\Lambda, L^{*})$ is nondegenerate if the twisted group algebra $L^{\alpha}\Lambda$ is an $L$-central simple algebra.
\end{rem}

We can now state our first structure theorem.

\begin{thm} \label{structure theorem 1}

Let $D$ be a division algebra, finite dimensional over its center $K$. Suppose $D$ is faithfully graded by a finite group $G$.  Let $D_{e}$ be the identity component of $D$ and denote by $L$ its center.  Assume $K$ is contained in $D_e$ and that $K$ contains $\mu_{n_{G}}$, $n_{G} = exp(G)$. Then the following conditions hold:

\begin{enumerate}
\item

Conjugation on $D_{e}$ by nonzero homogeneous elements $x_{g} \in D_{g}$ determines a $G$-action  on $L$, the center of $D_e$.  Furthermore, $K = L^{G}$ and hence if we denote by $H$ the kernel of the action, $L/K$ is a $G/H$-Galois extension.

\item

The group $H$ is abelian and is isomorphic to  $A \times A$, for some abelian group $A$.

\item

There exists a $2$-cocycle $\alpha \in Z^{2}(H, L^{*})$ such that $D_{H} = \sum_{h\in H}D_{h} \cong D_{e}\otimes_{L}L^{\alpha}H$. Furthermore, the cocycle $\alpha$ is nondegenerate, that is $L^{\alpha}H$ is an $L$-central simple algebra. The subalgebras  $L^{\alpha}H$ and  $D_{H}$ have center $L$.

\item

The cohomology class $[\alpha] \in H^{2}(H, L^{*})$ is $G/H$-invariant.

\item
The algebra $D$ is a $G/H$-crossed product. Its degree is $d\sqrt{ |H|}[G:H]$, where $d=\sqrt{dim_{L}(D_{e})}$. 

\end{enumerate}

\end{thm}

Our second structure theorem deals with the general case; we drop the assumption that $K$, the center of $D$,  is contained in the identity component and do not assume  $K$ is $G$-graded.

\begin{thm}\label{general structural theorem}

Let $D$ be a division algebra, finite dimensional over its center $K$. Suppose $D$ is faithfully graded by a finite group $G$. Let $D_{e}$ be the identity component of $D$ and denote by $L$ its center. Let $K_{0} = K \cap L$ denote the central $e$-homogeneous elements of $D$. Assume $K_{0}$ contains $\mu_{n_{G}}$, $n_{G} = exp(G)$. With these conditions the following hold:

\begin{enumerate}
\item

Conjugation on $D_{e}$ with nonzero homogeneous elements $x_{g} \in D_{g}$ determines a $G$-action on $L$ with $K_{0} = L^{G}$. If we denote by $H$ the kernel of the action, then $L/K_{0}$ is a $G/H$-Galois extension.



\item

The algebra $D$ is finite dimensional over  $K_{0}$.

\item

The group $H$ is abelian.

\item

There exists a $2$-cocycle $\alpha \in Z^{2}(H, L^{*})$ such that $D_{H} = \sum_{h\in H}D_{h} \cong D_{e}\otimes_{L}L^{\alpha}H$. Hence, if we denote $L_{1} = Z(L^{\alpha}H)$, then $L_1$ is also the center of the division algebra $D_{H}$.

\item

Conjugation on $D_{H}$ by nonzero homogeneous elements induces an action of $G$ on $L_{1}$ with fixed field $K$, the center of $D$. Furthermore, the extension $L_{1}/K$ is $G/H$-Galois.

\item

The cohomology class $[\alpha] \in H^{2}(H, L^{*})$ is $G/H$ invariant.  If we let $S=\{s\in H:x_s\in L_1\}\ \ (=\{s\in H: \alpha(s,h)=\alpha(h,s) \ \ \hbox{for all} \ \ h\in H\})$,  then $S$ is a subgroup of $H$ and  $L_1=L^\alpha S$.  The group $H/S$ is of central type, hence isomorphic to $A\times A$ for some abelian group $A$.

\item
The division algebra $D$ is a $G/H$-crossed product over the algebra $D_{H}$.  Its degree is $d\sqrt{[H:S]}[G:H]$, where $d = \sqrt{dim_{L}(D_{e})}$.  

\end{enumerate}

\end{thm}

In the next theorem we determine  under what circumstances the center $K$ (in Theorem~\ref{general structural theorem} above) is $G$-graded. By statement (5) of the theorem,  $K$ is a subfield of  $D_{H}$ and hence if $K$  is $G$-graded, it is necessarily $H$-graded.

We continue the notation as in Theorem~\ref{general structural theorem}. For every $h \in H$ we let  $x_{h}$ be a representative in $L^{\alpha}H$  (i.e.,  any nonzero element in $(L^{\alpha}H)_{h}$). Let $S = \{h \in H: x_{h}x_{h'} = x_{h'}x_{h}\ \   $for all$\ \  h' \in H \}$, that is, $S$ is the set of elements in $H$ such that $x_s\in L_1$, the center of $L^\alpha H$.   We {\it claim} that in fact $L_1=L^\alpha S$:  It is clear that  $L^\alpha S\subseteq L_1$.  Conversely if $z=\sum_{h\in H}l_hx_h$ lies in $L_1$,  then for all $r\in H$, $x_rz=zx_r$.  But $H$ is abelian and $x_r$ commutes with the elements of $L$.  It follows that if $l_h\not=0$  then $h\in S$.

Note that because $G$ acts on $L_{1}$, it acts also on $S$,  so $S$ is a normal subgroup of $G$.

\begin{thm} \label{graded center}
The center $K$ is $G$ graded if and only if $S$ is central in $G$.
Moreover if $S$ is central in $G$ then $K$ is a twisted group algebra $K_0^{\tilde \alpha} S$,  where $\tilde\alpha: S\times S\rightarrow K_{0}^{*}$ is a two-cocycle cohomologous to $\alpha$ in $H^2(S,L^{*})$.
\end{thm}

The most substantial part of this paper is devoted to proving a converse to the structure theorems. To this end we adopt the following terminology:

Let $D$ be a finite dimensional division algebra with center $K$ and suppose $D$ is faithfully graded by a finite group $G$. Assume the $e$-center $K_{0}$ of $D$ contains a primitive $n_{G}$-th root of unity where $n_{G}$ is the exponent of $G$. By the structure theorems this gives rise to

\begin{enumerate}

\item

a group extension with abelian kernel
$$1 \rightarrow H \rightarrow G \rightarrow G/H \rightarrow 1,$$

\item a cohomology class $[\alpha] \in H^{2}(H, L^{*})^{G/H}$, where $L = Z(D_{e})$,

\noindent and

\item a positive integer $d$, the degree of $D_e$ over its center.

\end{enumerate}

\medskip

Let $M(H)$ denote the Schur multiplier of $H$ and let $\phi_{[\alpha]}\in Hom(M(H), L^{*})$ be the map corresponding to $[\alpha]$ by means of the Universal Coefficients Theorem. By  naturality, the map $\phi_{[\alpha]}$ is $G$-invariant. Furthermore, because the values of $\phi_{[\alpha]}$ are contained in the group $\mu_{n_{H}}$, and these are assumed to be in $K_{0}$, we have $\phi_{[\alpha]}\in Hom(M(H), \mu_{n_{H}})^{G/H}$, where the action of $G$ on $\mu_{n_{H}}$ is trivial. Because $H$ is abelian the Schur multiplier can be identified with the wedge product $H\wedge H$ and if $h_1, h_2\in H$,  then $\phi_{[\alpha]}(h_1\wedge h_2)=\alpha(h_1,h_2)\alpha(h_2,h_1)^{-1}$. In the twisted group algebra $L^\alpha H=\sum_{h\in H}Lx_h$ one sees easily that for all $h_1,h_2\in H$, the commutator $x_{h_1}x_{h_2}x_{h_1}^{-1}x_{h_2}^{-1}$ is a root of unity and we also have $\phi_{[\alpha]}(h_1\wedge h_2)=x_{h_1}x_{h_2}x_{h_1}^{-1}x_{h_2}^{-1}$.
Note that if $F$ is an algebraically closed field containing $K_{0}$ then in fact $H^{2}(H, F^{*})^{G/H}\cong Hom(M(H), \mu_{n_{H}})^{G/H}$, and hence it makes sense to speak of the nondegeneracy of $\phi_{[\alpha]}\in Hom(M(H), \mu_{n_{H}})^{G/H}$ (see Remark \ref{central type remark}). Here, the action of $G$ on $F^{*}$ and on $\mu_{n_{H}}$ is the trivial action.

Letting  $[\beta] \in H^{2}(G/H, H)$ denote the cohomology class determined by the group extension given above, we see that
any finite dimensional division algebra faithfully graded by the group $G$ determines uniquely an ordered triple $([\beta], \phi, d)  \in H^{2}(Q, H) \times Hom(M(H), \mu_{n_{H}})^{Q}\times \mathbb N$ where $Q = G/H$ and $\mathbb N=\{1,2,3,\dots \}$ denotes the set of natural numbers.

\begin{defn} \label{realizable triple definition} Let $Q$ be a finite group and $H$ an abelian group. Let $n_{H} = exp(H)$. We say the triple $([\beta], \phi, d) \in H^{2}(Q, H) \times Hom(M(H), \mu_{n_{H}})^{Q}\times \mathbb N$ is {\it realizable} if there exists a finite dimensional division algebra, faithfully $G$-graded, over a field of characteristic zero containing $\mu_{n_{H}}$, that yields that triple. Here, the group $G$ is the group determined by the class $[\beta]$.

\end{defn}

Here is the main result of the paper. Notation as above.

\begin{thm} \label{Main realizable triple theorem}
If $Q$ is a finite group and $H$ is an abelian group, then every ordered triple  $([\beta], \phi,d) \in H^{2}(Q, H) \times Hom(M(H), \mu_{n_{H}})^{Q}\times \mathbb N$ is realizable. Moreover, it is realizable over a field containing an algebraically closed field of characteristic zero. If $\phi$ is nondegenerate then the center of every realization will lie in the $e$-component of its grading.

\end{thm}

We will say that an extension
$$\beta: 1 \rightarrow H \rightarrow G \rightarrow Q \rightarrow 1$$ with an abelian kernel is realizable if there exists a $Q$-invariant map $\phi: M(H) \rightarrow \mu_{n_{H}}$ such that the triple $([\beta], \phi,1) $ is realizable.  We obtain the following corollary  (see Cuadra and Etingof result (\cite{CuadraEtingof}, Theorem $2.1$)).


\begin{cor} \label{realizable extension corollary}
Every extension
$$\beta: 1 \rightarrow H \rightarrow G \rightarrow Q \rightarrow 1$$ with an abelian kernel is realizable by  a division algebra of degree $[G:H]$.
\end{cor}

\begin{proof} If $\phi$ is trivial then $\phi$ is $Q$-invariant and so by the theorem the triple $([\beta],[\phi],1)$ is realizable by a division algebra $D$.  Moreover the $e$-component of $D$ is a field $L$ and, because $\phi$ is trivial, the $H$ component is a field extension $L_1$ of $L$.  The group $Q$ acts faithfully on $L_1$ and the fixed field is the center $K$ of $D$.  It follows that $D$ is a (classical) crossed-product algebra of degree $Q$. \end{proof}





\begin{example} \label{examples quaternion and dihedral}
In \cite{CuadraEtingof} Cuadra and Etingof give an example of a finite dimensional division algebra  graded by $Q_{8}$, the quaternion group of order $8$. In their example the center is not graded. In the following we describe the possible grading structure on division algebras over fields which contain a primitive $4$th root of unity, by the group $Q_{8}$ and by the group $D_4$, the dihedral group of order $8$.  As above $D$ will denote the division algebra,  $K$ its center and $K_0$ its $e$-center.   We will assume    $D_{e} = L$,  so $d=1$ $($See the remarks at the beginning of section 3$)$ and  $L_1$  denotes the center of the $H$-component.

Let $$1 \rightarrow H \rightarrow Q_{8} \rightarrow Q_{8}/H \rightarrow 1$$

\begin{enumerate}

\item
If $H = \{1\}$ we obtain a $Q_{8}-$crossed product division algebra of degree $8$.
\item

If $H \cong \mathbb{Z}_{2}$,  the $e$-component $L$ of $D$  is a $\mathbb{Z}_2\times \mathbb{Z}_2$ Galois extension of the $e$-center $K_{0}$.   The $H$-component $D_{H}$ is a field, so equal to $L_{1}$ and $K = L_1^{G/H}=K_0^{\alpha}H$ is graded, a quadratic extension of $K_0$. The division algebra $D$ is a $\mathbb{Z}_2\times \mathbb{Z}_2$-crossed product for the extension $L_1/K$.
\item

If $H \cong \mathbb{Z}_{4}$, the group $S$ of Theorem \ref{general structural theorem} equals $H$ $($because $H$ is cyclic$)$.  The $e$-component $L$ is an extension of degree 2 over  $K_0$ and the $H$-component is a field extension $($so equal to $L_1$$)$ of $L$ of degree 4.   The center $K= L_1^{G/H}$ is ungraded,  an extension of degree 4 over $K_0$.  The division algebra $D$ is  a quaternion algebra  over  $K$.  This is the example of Cuadra and Etingof.

\bigskip

\end{enumerate}

Now let $$1 \rightarrow H \rightarrow D_4 \rightarrow D_4/H \rightarrow 1$$

\begin{enumerate}

\item

If $H = \{1\}$ we obtain a $D_4$ crossed product division algebra of degree $8$.
\item

If $H \cong \mathbb{Z}_{2}$ the $e$-component  $L$ of $D$ is a $\mathbb{Z}_2\times \mathbb{Z}_2$ Galois extension of the $e$-center $K_0$.   The $H$-component $D_H$ is a field, so equal to  $L_{1}$ and $K = L_1^{G/H}=
K_0^{\alpha}H$ is graded, a quadratic extension of $K_0$.     The division algebra $D$ is a $\mathbb{Z}_2\times \mathbb{Z}_2$-crossed product for the extension $L_1/K$.

\item
If $H \cong \mathbb{Z}_{2} \times \mathbb{Z}_{2}$  and the map $\phi$ is trivial, then  the $e$-component  $L$ is an extension of degree 2 over $K_0$ and the $H$-component is a field (so equal to $L_1$), a biquadratic extension of $L$.    The division algebra $D$  is a  quaternion algebra  over the  nongraded center $K=L_1^{G/H}$ which is of degree 4 over $K_0$. Because the sequence is split,  $D$ is isomorphic to $D_0\otimes_{K_0}K$,  where $D_0$ is a quaternion algebra over $K_0$: Indeed, if $D_{4} = \{\sigma, \tau: \sigma^{4} = \tau^{2} = e, \sigma^{-1} \tau = \tau \sigma\}$ and $H = \{e, \sigma^{2}, \tau, \sigma^{2}\tau\}$, the element $\sigma\tau$ is of order $2$ and acts nontrivially on $L$. Thus $L$ and $\sigma\tau$ generate a quaternion subalgebra $D_0$ over $K_0$ and $D\cong D_0\otimes_{K_0}K$.

\item

If $H \cong \mathbb{Z}_{2} \times \mathbb{Z}_{2}$ and the map $\phi$ is nontrivial,  then $\phi$ is the only nontrivial element of  $Hom(M(H), \mu_{2})$  and so   $\phi$ is  $G$-invariant.   The $e$-component $L$   has degree 2 over $K_0$.  The $H$-component of $D$ is a quaternion algebra with center $L$.  Therefore $L_1=L$ and $K=K_0$.  Because the sequence is split,  $D\cong Q_1\otimes_K Q_2$ is a tensor product of two quaternion algebras over the center $K$. Indeed, if $D_{4} = \{\sigma, \tau: \sigma^{4} = \tau^{2} = e, \sigma^{-1} \tau = \tau \sigma\}$ and $H = \{e, \sigma^{2}, \tau, \sigma^{2}\tau\}$, the element $\sigma\tau$ is of order $2$ and acts nontrivially on $L$. Thus it generates a quaternion subalgebra $Q_{1}$ over $K$ which can be factored from $D$.

\item
If $H \cong \mathbb{Z}_{4}$, the group $S$ of Theorem \ref{general structural theorem} equals $H$ $($because $H$ is cyclic$)$.  The $e$-component   $L$ is  of degree 2 over $K_0$ and the $H$-component is a field extension of $L$ of degree 4.  So $L_1$ equals the $H$-component.    The center $K=L_1^{G/H}$ is ungraded,  an extension of degree 4 over $K_0$.  The division algebra $D$ is  a quaternion algebra  over  $K$ and because the extension is necessarily split, as in example (3)  there is a quaternion algebra $D_0$ over $K_0$ such that $D\cong D_0\otimes_{K_0}K$.

\bigskip

\end{enumerate}



\end{example}

Next we address the following problem: What are the $G$-graded simple algebras over an algebraically closed field of characteristic zero that admit a $G$-graded form division algebra?

Let $D$ be a finite dimensional division algebra faithfully $G$-graded with center $K$. Let $K_{0} = K \cap D_{e}$ as above. Suppose $K_{0}$ contains $\mu_{n_{G}}$, $n_{G} = exp(G)$. Extension of scalars $D_{F} = D\otimes_{K_{0}}F$ where $F$ is any algebraically closed field containing $K_{0}$ yields a finite dimensional $G$-graded simple algebra in which $F$ will be the field of central homogeneous elements of degree $e$. Such algebras have been classified by means of elementary and fine gradings by Bahturin, Sehgal and Zaicev \cite{BSZ}. Here is their result:

\begin{thm} \label{Bahturin Sehgal Zaicev theorem}
Let $A$ be a finite dimensional $G-$simple algebra over an algebraically closed field
$F$ of characteristic zero.  There is a subgroup $H$ of $G$, a two cocycle $\alpha\in Z^{2}(H,F^{*})$
$($the $H-$action on $F^{*}$ is trivial$)$ and $\mathfrak{g}=(g_{1},\ldots,g_{s}) \in G ^{(s)}$,
such that $A\cong F^{\alpha}H\otimes M_{s}(F)$, where $($by this identification$)$ the $G-$grading on $A$ is given
by
\[
A_{g}=\mbox{span}_{F}\left\{ u_{h}\otimes e_{i,j}:g=g_{i}^{-1}hg_{j}\right\} .
\]
\end{thm}
Furthermore, we may replace the elements in the $s$-tuple by right H-cosets representatives and also, by permuting elements of the $s$-tuple we may assume equal representatives
are adjacent to each other (see \cite{BSZ0}, \cite{ElduqueKotchetov}, \cite{AljHaile}).

Unlike the ungraded case, namely matrix algebras over $F$, not every finite dimensional $G$-graded simple algebra over $F$ admits a division algebra form which is $G$-graded. In fact it may not admit a form which is a $G$-graded division algebra (i.e. nonzero homogeneous elements are invertible) (see \cite{AljadeffKarasikVerbally}, Theorems $1.8$ and $1.12$).

We have the following result.

\begin{thm} \label{Simple algebras forms} Notation as above.
Let $A$ be a finite dimensional $G$-graded simple algebra over an algebraically closed field $F$ of characteristic zero. Then $A$ admits a division algebra $G$-graded $K_{0}$-form, for some field $K_{0}$ which contains $\mu_{n_{G}}$, $n_{G} = exp(G)$, if and only if the following hold:

\begin{enumerate}
\item

$H$ is abelian.

\item

Every coset of $H$ is represented in the $s$-tuple (elementary grading). Moreover the number of representatives is equal for all cosets.

\item

$H$ is normal in $G$.

\item

The cohomology class $[\alpha] \in H^{2}(H, F^{*})$ is $G$ invariant where the action of $G$ on $F$ is trivial.

\end{enumerate}
Furthermore, with the condition above, $A$ admits a division algebra $G$-graded $K_{0}$-form where $K_{0}$ contains an algebraically closed field.
\end{thm}

\section{proof of the structure theorems}

In this section we prove Theorems \ref{structure theorem 1}$-$\ref{graded center}. We start with Theorem \ref{structure theorem 1}.

\begin{proof}
Let $D$ be a finite dimensional algebra over its center $K$ and suppose $D$ is graded by a finite group $G$. We are assuming the center $K$ is contained in  $D_{e}$, the identity component of $D$. Let $L = Z(D_{e})$. Conjugation on $D_{e}$ with nonzero homogeneous elements induces an action of $G$ on $L$ and we denote by $H$ the kernel of this action. We have $K = L^{G} = L^{G/H}$.
We {\it claim} that $D_{H} \cong D_{e} \otimes_{L}L^{\alpha}H$ for some $\alpha \in Z^{2}(H, L^{*})$: For all $h \in H$, let $u_{h}$ be a nonzero homogeneous element of degree $h$. Because  $h$ acts trivially on $L$, and $D_{e}$ is finite dimensional over $L$, conjugation by $u_{h}$ determines an inner automorphism of $D_{e}$ and hence, by the Skolem-Noether Theorem,  there is a (nonzero) element $\theta_h$ in $D_{e}$ such that $x_{h} = u_{h}\theta_{h}^{-1}$ centralizes $D_{e}$. Let $h$ and $h'$ in $H$. Clearly $x_{h}x_{h'} = \alpha(h,h')x_{hh'}$ where $\alpha(h,h') \in D_{e}^{*}$,  but because $x_{h}$  and $x_{h'}$ centralize $D_{e}$,  we have $\alpha(h,h') \in L^{*}$.  It follows that $\alpha : H\times H \rightarrow L^{*}$ is a $2$-cocycle in $Z^{2}(H,L^{*})$. Finally, because the set $\{x_{h}\}_{h \in H}$ is linearly independent over $D_{e}$ (and in particular over $L$) we obtain $D_{H} = D_{e}\otimes_{L}L^{\alpha}H$,  as desired.

Next we will  show $L$ is the center of the twisted group algebra $L^{\alpha}H$:  If $L_{1} = Z(L^{\alpha}H) = Z(D_{H})$,  conjugation by nonzero homogeneous elements of $D$ induces an action of $G$ on $L_{1}$ with fixed field $K$. But $H$ acts trivially on $L_{1}$ and so $L_{1}^{G/H} = K$. Because  $L \subseteq L_{1}$ and the action of $G/H$ is faithful, we obtain $L=L_{1}$, as desired.
It follows that the $2$-cocycle $\alpha$ is nondegenerate and $H$ is of central type.

We claim next that $H$ is abelian: Because this may be of independent interest we state it as a separate proposition. Its proof is presented right after the end of the proof of Theorem \ref{structure theorem 1}.

\begin{prop} \label{abelian group twisted group algebra proposition}
Let $L$ be a field of characteristic zero and let $H$ be a finite group of exponent $n_{H}$. Let $L^{\alpha}H$ be a twisted group algebra of $H$ over $L$ where $\alpha \in H^{2}(H,L^{*})$ and the action of $H$ on $L$ is trivial. Suppose the field $L$ contains the group $\mu_{n_{H}}$ of $n_{H}$th roots of unity. If $L^{\alpha}H$ is a division algebra then $H$ is abelian.
\end{prop}
Note that the commutativity of the group $H$ in Theorem \ref{structure theorem 1} follows because $\mu_{n_H} \leq \mu_{n_{G}} \leq k^{*} \leq L^{*}$.

Next we will show that the class $[\alpha] \in H^{2}(H, L^{*})$  is $G$ invariant, where the action of $G$ on the cohomology is induced by the given action on $H$ and $L$. Let $y_{g}$ be a nonzero homogeneous element of degree $g \in G$. If $h\in H$,  then   $y_{g}x_{h}y_{g}^{-1}$ lies in $D_H$ and commutes with the elements of  $D_{e}$.  It follows that conjugation by $y_g$ stabilizes $L^\alpha H$.    We need to show the cocycle $g(\alpha)$ determined by $g(\alpha)(h,h') = y_{g}^{-1}\alpha (ghg^{-1}, gh'g^{-1})y_{g}$ is cohomologous to  $\alpha$.
Write $x_{h}x_{h'} = \alpha(h,h')x_{hh'}$. Conjugating both sides by $y_{g}^{-1}$ we obtain
$$
y_{g}^{-1}x_{h}y_{g}y_{g}^{-1}x_{h'}y_{g} = y_{g}^{-1}\alpha(h,h')y_{g}y_{g}^{-1}x_{hh'}y_{g}.
$$

Because $y_{g}^{-1}x_{h}y_{g} \in D_{g^{-1}hg}$ and centralizes $D_{e}$, we have $y_{g}^{-1}x_{h}y_{g} = \gamma_{h}x_{g^{-1}hg}$ where $\gamma = \gamma(g): H \rightarrow L^{*}$. Note that the elements of $D_{e}$ centralize $x_{h}$, so the value $\gamma_{h}$ does not depend on the choice of $y_{g} \in D_{g} = D_{e}y_{g}$.

Thus the equation above yields
$$
\gamma_{h}x_{g^{-1}hg}\gamma_{h'}x_{g^{-1}h'g} = y_{g}^{-1}\alpha(h,h')y_{g}\gamma_{hh'}x_{g^{-1}hh'g}
$$

and so
$$
\gamma_{h}\gamma_{h'}\alpha(g^{-1}hg, g^{-1}h'g)x_{g^{-1}hh'g} = y_{g}^{-1}\alpha(h,h')y_{g}\gamma_{hh'}x_{g^{-1}hh'g}.
$$
We conclude that
$$
\gamma_{h}\gamma_{h'}\alpha(g^{-1}hg, g^{-1}h'g) = y_{g}^{-1}\alpha(h,h')y_{g}\gamma_{hh'},
$$
showing the cocycles are cohomologous.

The algebra $D_H$ is an $L$-central division of degree $\sqrt{\dim_{L}(D_{e})|H|}$ and $D$ is a $G/H$-crossed product.
Moreover $K\subseteq L\subseteq D_H\subseteq D$ and $[D:D_H]=[L:K]=|G/H|$,  so
$deg(D)=\sqrt{\dim_K(D)}= deg(D_H)|G/H|=  \sqrt{\dim_{L}(D_{e})|H|}[G:H]$.
\end{proof}

We proceed with the proof of Proposition~\ref{abelian group twisted group algebra proposition}.
\begin{proof}
Our objective in the first part of the proof is to show the group $H$ must be nilpotent. This will reduce the problem to $p$-groups, a case which is treated in Lemma \ref{abelian p-group lemma}.

\begin{lem} \label{nilpotency}
Let $H$ be a finite group. Suppose
\begin{enumerate}
\item
The commutator subgroup $H'$ is nilpotent.
\item
If $p$ and $q$ are different primes, every $p$-element $z \in H$ commutes with every $q$-element $w \in H'$.
\end{enumerate}
Then $H$ is nilpotent.
\end{lem}
\begin{proof}
We show every Sylow $p$-subgroup of $H$ is normal. Let $y \in H$ and $x\in P$, a Sylow $p$-subgroup of $H$. Write $yxy^{-1} = cx$ with $c \in H'$ and by the nilpotency of $H'$ let $c = c_{1}c_{2}$ where $c_{1}$ and $c_{2}$ are elements in $H'$ of orders $p'$ (prime to $p$) and $p$ respectively. Because $yxy^{-1}$ commutes with $c_{1}$ it commutes with $c_{2}x$. We claim $c_{2}x$ is a $p$-element: Indeed, because $(H')_{p}$, the $p$-Sylow subgroup of $H'$, is characteristic in $H'$, it is normal in $H$ and hence it is contained in every $Sylow$ $p$-subgroup of $H$ proving the claim. But $yxy^{-1}$ and $c_{2}x$ commute and hence $yxy^{-1}(c_{2}x)^{-1} = c_{1}$ is both a $p$ and a $p'$ element. It follows that $yxy^{-1} = c_{2}x$ and because $c_{2} \in  (H')_{p} \leq P$, we have $c_{2}x \in P$ and the lemma is proved.
\end{proof}
We now resume the proof of the proposition. Let $\Gamma$ be the group of trivial units in $L^{\alpha}H$, that is the group of elements of the form $l_hu_{h}$, where $ h\in H$ and $l_h\in L^{*}$. Because $\Gamma$ is center by finite, it follows by a theorem of Schur that the commutator subgroup $\Gamma'$ of $\Gamma$ is finite. Note that $\Gamma/L^{*} \cong H$ and $\Gamma'/(L^{*}\cap \Gamma) \cong H'$. Invoking Amitsur's classification of finite subgroups of $D^{*}$ where $D$ is an arbitrary division algebra (\cite{Amitsur}), we have that $\Gamma'$ is one of the following: ($1$) metacyclic (including cyclic) ($2$) $\hat{A}_{4}$ the binary tetrahedral group of order $24$ ($3$) $\hat{S}_{4}$ the binary octahedral group of order $48$ ($4$) $\hat{A}_{5}$ the binary icosahedral group of order $120$.

We first show that under our assumptions cases $2-4$ are not possible. Then in case ($1$) we will show that the group $H$ satisfies the conditions of Lemma \ref{nilpotency} and hence is nilpotent.

For cases $2-4$ we apply (\cite{AljadeffSonn} Section $2$, Main Lemma). In that lemma it is proved that $\Gamma'$ contains a subgroup $U$, normal in $\Gamma$, with the property that $U$ is isomorphic to $Q_{8}$, the quaternion group of order $8$. We \textit{claim} that $U \cap L^{*} \cong \mathbb{Z}_{2}$: If not $U \cap L^{*} = \{e\}$ and so $U$ embeds in $H$. By the Universal Coefficient Theorem together with the fact that the Schur multiplier $M(U)$ is trivial, we have the group algebra $LU' \subseteq L^{\alpha}H$, showing $L^{\alpha}H$ cannot be a division algebra. This proves the claim. It follows that because $L^{\alpha}H$ is a division algebra,  the group $U$ spans $(-1,-1)_{2, L}$, the Hamilton quaternions over $L$. Next we claim the algebra $(-1,-1)_{2, L}$ is split contradicting $L^{\alpha}H$ is a division algebra. To see this, note that in cases $2-4$ the $3$-Sylow subgroup of $\Gamma'$ is not central, and so $3\mid exp(H)$. It follows by our assumption that $\zeta_{3} \in L^{*}$ and so $(-1,-1)_{2, L}$ is split.

We turn now to case ($1$) when $\Gamma'$ is metacyclic. The treatment in this case is divided into two cases, denoted $(1.1)$ and $(1.2)$ in the notation of (\cite{AljadeffSonn} Section $2$, Main Lemma). In case $(1.1)$ it is shown in that lemma that $\Gamma'$ must be cyclic. In case $(1.2)$ it is shown that $\Gamma' \cong C_{t} \times U$ where $C_{t}$ is cyclic of odd order and $U$ is isomorphic to $Q_{8}$. Furthermore, as above $U \cap L^{*} \cong \mathbb{Z}_{2}$ and $U$ spans $(-1,-1)_{2, L}$, the Hamilton quaternions over $L$. The proof of the nilpotency of $H$ will be completed if we show that in both cases $(1.1)$ and $(1.2)$ the group $H$ satisfies the conditions of Lemma \ref{nilpotency}. The nilpotency of $H' \cong L^{*}\Gamma'/L^{*}$ is evident. For the second condition we assume first $\Gamma'$ is cyclic: Because a $q$-Sylow subgroup $(\Gamma')_{q}$ of $\Gamma'$ is characteristic in $\Gamma'$, it is normal in $\Gamma$. But the presence of $\zeta_{q} \in L^{*}$ implies the inner automorphisms on $(\Gamma')_{q}$ are of $q$-power order which means that for $p \neq q$, any $p$-element of $H$ centralizes $(\Gamma')_{q}$. Suppose now $\Gamma' \cong C_{t} \times U$ where $t$ is odd. As above every $q$-Sylow subgroup $(\Gamma')_{q}$ is characteristic in $\Gamma'$ and hence normal in $\Gamma$. If $q$ is odd the argument is as above. Suppose $q = 2$. The automorphism group of $U$ is $S_{4}$ and the only way the second condition of the lemma can fail is if there exists an element $y \in H$ whose order is divisible by $3$ that acts nontrivially on $L^{*}U/L^{*}$ by conjugation. But this is not possible because in that case $\zeta_{3} \in L^{*}$ and the algebra $(-1,-1)_{2, L}$ is split. \end{proof}

As promised we complete the proof of Proposition ~\ref{abelian group twisted group algebra proposition} with the following lemma.
\begin{lem}\label{abelian p-group lemma}
Let $H$ be a finite $p$-group, say of exponent $p^{r}, r\geq 1$. Let $L$ be a field of characteristic zero containing $\zeta_{p^{r}}$, a primitive $p^{r}th$ root of unity. If a twisted group algebra $L^{\alpha}H$ is a division algebra, the group $H$ must be abelian.
\end{lem}

\begin{proof}
Notation as above. Let $p^{s}$ the order of $\Gamma'$. We \textit{claim} $\Gamma'$ is cyclic. Indeed, by the analysis given in the first part of the proof of Proposition \ref{abelian group twisted group algebra proposition} the only way this can be false is if $\Gamma' \cong Q_{8}$ which spans $(-1,-1)_{2, L}$ over $L$. The field $L$ cannot contain $\mu_{4}$ because otherwise $(-1,-1)_{2, L}$ is split. Therefore the exponent of $H$ must be $2$ and hence $H$ is abelian in this case. This proves the claim.
So, let $\Gamma' = \langle z \rangle$ where $z$ has order $p^{s}$. Because $\Gamma'$ is a cyclic $p$-group the element $z$ is a single commutator and so there are elements $x$ and $y$ such that $yxy^{-1} = zx$. Following the assumption that $L$ contains $\zeta_{p^{r}}$, we let $\zeta_{p^{r_{1}}}$, $r_{1} \geq r$ be the maximal $p$-power root of unity in $L^{*}$. Note that we may assume such a maximum exists, for otherwise the finite cyclic group $\Gamma'$ would be contained in $L^{*}$ showing $H$ is abelian. For the proof of the lemma we may assume the integers ${r, r_{1}, s}$ satisfy the following relations: $s \geq r_{1} \geq r \geq 1$ and $s - r_{1} \leq r$.
We will show that $s \leq r$. In that case $\Gamma' \subset L^{*}$, so $H$ is abelian.

Let $\alpha$ be a nonnegative integer such that $yzy^{-1} = z^{\alpha +1}$. Because $z^{p^{s-r_{1}}}$ is central, we have on the one hand $yz^{p^{s-r_{1}}}y^{-1} = z^{p^{s-r_{1}}}$ and on the other hand
$$yz^{p^{s-r_{1}}}y^{-1} = (z^{\alpha +1})^{p^{s-r_{1}}} = z^{p^{s-r_{1}}(\alpha + 1)}.$$
This implies $p^{s}$ divides $p^{s-r_{1}}\alpha$, or equivalently that $p^{r_{1}}$ divides $\alpha$. Hence we can write $\alpha = k_{0}p^{t}$ where $(k_{0}, p) = 1$ and $t \geq r_{1}.$
Now, applying conjugation by $y$, $p^{r}$ times on $x$, we obtain
$$x =  y^{p^{r}}xy^{-p^{r}} = z^{{(\alpha + 1)}^{p^{r}-1}}z^{{(\alpha + 1)}^{p^{r}-2}}\cdots z^{(\alpha + 1)}z\cdot x
$$
and so $z^{({(\alpha + 1)^{p^{r}} - 1})/\alpha} = 1$ in $\Gamma$. This shows $p^{s}$ divides $M_{1} = {({(\alpha + 1)^{p^{r}} - 1})/\alpha}.$
Claim: $s = r_{1}$. Write
$$
M_{1} = {({(\alpha + 1)^{p^{r}} - 1})/\alpha} = p^{r} + (p^{r}(p^{r} -1)/2)\alpha + a_{3}\alpha^{2}+\cdots +{p^{r}}\alpha^{p^{r} -2} + 1 \alpha^{p^{r} -1}
$$
where $a_{3}$ is a binomial coefficient. Because $\alpha$ is divisible by $p^{t}$ where $t \geq r_{1} \geq r \geq 1$, we have
$$M_{1} \equiv p^{r} + (p^{r}(p^{r} -1)/2) \alpha \mod (p^{r+1}).
$$
Now, if $p$ is odd or $p = 2$ and $r \geq 2$, $M_{1}\equiv p^{r}$ $\mod (p^{r+1})$ and so the only way $M_{1}$ can be divisible by $p^{s}$ is if $s \leq r$.
We are left with the case where $exp(H) = 2$, which implies $H$ is abelian. This completes the proof of the lemma and also of the proposition.
\end{proof}

In the next example we show the presence of $\mu_{n_{H}}$ in $L$ is necessary to ensure the group $H$ is abelian (see \cite{AljHaile_0}).
\begin{example} \label{nonabelian examples}
Let $\mathbb{Z}_{p^{2}} \rtimes \mathbb{Z}_{p^{2}} = \langle b, \sigma: \sigma b\sigma^{-1} = b^{p+1}\rangle$ and $L = \mathbb{Q}(x,y)(\zeta_{p})$ where $p$ is an odd prime. There exists a cohomology class $\alpha \in H^{2}(H, L^{*})$
such that the corresponding twisted group algebra $L^{\alpha}H$ is a division algebra.
\end{example}

We proceed to prove Theorem \ref{general structural theorem}, that is, we drop the assumption that the center is contained in the $e$-component.

\begin{proof}

We have $D = D_{e} \oplus D_{g_2} \oplus \cdots \oplus D_{g_{n}}$ and $L=Z(D_{e})$. As above, the group $G$ acts on $L$ and if $H$ denotes the kernel of the action, then  $K_{0} = L^{G/H}$. We want to show, as above, that we can find representatives $x_{h}, h\in H$, which centralize $D_{e}$ so that we have $D_{H} \cong D_{e}\otimes_{L}L^{\alpha}H$, for some $\alpha \in Z^{2}(H,L^{*})$. In order to apply Skolem-Noether we need to know $D_{e}$ is finite dimensional over its center $L$. We are given that $D$ is finite dimensional over $K$.

Claim: the center $K$ is contained in $D_{\tilde{H}}$, where $\tilde{H}$ is the subgroup of $H$ consisting of all elements $g \in G$ for which there exists $u_{g} \in D_{g}$ that induces an inner automorphism on $D_{e}$ by conjugation. Note that in that case there exists (as above) an element $x_{g} \in D_{g}$ that centralizes $D_{e}$. To prove the claim let $z\in K$.  Then there are elements $\alpha_e, \alpha_{g_2}, \dots, \alpha_{g_n} \in D_e$ such that
$$
z = \alpha_{e}u_{e} + \alpha_{g_2}u_{g_2} + \cdots + \alpha_{g_n}u_{g_n}.
$$
We show first $\alpha_{g}=0$ if $g \not\in H$. Conjugation by elements of $L$ centralizes $z$ and preserves all homogeneous components. But if $g\not\in H$,   we can find $l_{g} \in L^{*}$ which does not commute  with  $u_{g}$, so $\alpha_{g}=0$, as desired. Suppose now that there is an element $h\in H$ such that conjugation by $u_{h}$ does not give an inner automorphism  on $D_{e}$. Need to show $\alpha_{h} = 0$. Conjugation by any nonzero element of $D_{e}$ fixes $z$ and so must centralize $\alpha_{h}u_{h}$. This shows conjugation of $D_{e}$ by $\alpha^{-1}_{h}$ and $u_{h}$ determine the same automorphism and hence the action by $u_{h}$ is inner if  $\alpha$ is nonzero.  This proves that $K$ is contained in $D_{\tilde{H}}$,  as desired.  It follows that $K$ is contained in the center of $D_{\tilde{H}}$ and that $D_{\tilde{H}}$ is finite dimensional over its center.

For every element $h \in \tilde{H}$ there is a nonzero element $x_{h}\in D_h$ that centralizes $D_{e}$.   If $h_1,h_2\in \tilde{H}$ then the product $x_{h_1}x_{h_2}$ also centralizes $D_e$.  It follows that  $x_{h_1}x_{h_2}\in L^{*}x_{h_1h_2}$.  Therefore   the algebra generated over $L$ by $\{x_{h}: h \in \tilde{H}\}$ is isomorphic to a twisted group algebra $L^{\alpha}\tilde{H}$ where $\alpha : \tilde{H} \times \tilde{H} \rightarrow L^{*}$ is a $2$-cocycle. Now, if we let $L_{1} = Z(L^{\alpha}\tilde{H})$, we see that $L_{1}$, which  is a finite extension of  $L$,  is the center of $D_{\tilde{H}} \cong D_{e}\otimes_{L}L^{\alpha}\tilde{H}$.  But we have seen that $D_{\tilde{H}}$ is finite dimensional over its center  and so  we have that $D_{\tilde{H}}$ is finite dimensional over $L$.  This shows that $D_{e}$ is finite dimensional over its center $L$ and hence finite dimensional over $K_{0} = K\cap L=L^{G/H}$. Note that in fact $H = \tilde{H}$, because by Skolem-Noether, conjugation by every homogeneous element of $D_{H}$ is inner. We have proved $(2)$ and $(4)$.

The proof that $H$ is abelian and that the cocycle $\alpha$ is $G$-invariant (part $(3)$ and the first statement of part $(6)$) is the same as in Theorem \ref{structure theorem 1}.


For the proof of $(5)$ note that conjugation of $D_{H}$ by homogeneous elements induces an action of $G$ on $L_{1}$ whose kernel contains $H$. But then the kernel must be equal to $H$ because $L$ is contained $L_{1}$.

Proof of the second part of $(6)$: Let $S = \{h \in H: x_{h}x_{h'} = x_{h'}x_{h}\ \   $for all$\ \  h' \in H \}$, that is, $S$ is the set of elements $s\in H$ such that $x_s\in L_1$, the center of $L^\alpha H$.   We {\it claim} that in fact $L_1=L^\alpha S$:  It is clear that  $L^\alpha S\subseteq L_1$.  Conversely if $z=\sum_{h\in H}l_hx_h$ lies in $L_1$,  then for all $r\in H$, $x_rz=zx_r$.  But $H$ is abelian and $x_r$ commutes with the elements of $L$.  It follows that if $l_h\not=0$  then $h\in S$.

Because $L_1=L^\alpha S$ is the center of $L^\alpha H$ the group $H/S$ must be of central type, that is, isomorphic to $A\times A$ for some abelian group $A$.

Proof of $(7)$: The algebra $D_H$ is an $L_1$-central division of degree $deg(D_e)\sqrt{|H/S|}$  and $D$ is a $G/H$-crossed product.     Moreover $K\subseteq L_1\subseteq D_H\subseteq D$ and $[D:D_H]=[L_1:K]=|G/H|$,  so
$deg(D)=\sqrt{\dim_K(D)}= deg(D_H)|G/H|=   deg(D_e)\sqrt{H/S}|G/H|$, as desired.

\end{proof}


\medskip

We proceed to Theorem \ref{graded center}.

\begin{proof}
The set up is as in Theorem \ref{general structural theorem}. Recall that $L_{1} = Z(L^{\alpha}H)$ where $H$ is abelian and that $S$ is the subgroup of elements $s\in H$  such that   $x_{s}$ commutes with $x_{h}$ for all $h \in H$.  We have seen that  $L_{1} = L^{\alpha}S$ and that $S$ is normal in $G$.

Let $S_{1} = S \cap Z(G)$. Clearly $L_{2} = L^{\alpha}S_{1}$ is contained in $L_{1}$ and $L_2$ is normalized by $G$. Suppose $S_{1} = S$, that is, $S$ is central in $G$. We want to show the center $K$ is $H$-graded and in particular $G$-graded. Indeed, if $z\in K$, $z = \sum_{s \in S} \gamma_{s}x_{s}$, conjugation by nonzero homogeneous elements $y_{g}$ acts on one hand trivially on $z$ and on the other hand by multiplication of each $x_{s}$ by a nonzero scalar. Because the elements $x_{s}$ are linearly independent, the result follows.

Suppose now $S$ is not central in $G$, that is, $S_{1}$ is a proper subgroup of $S$. We want to show the center $K$ is not graded. 
Take $s_{0}\in S\setminus S_{1}$. Clearly $x_{s_{0}}$ is not in $K$.  For every  $g \in G$ we choose nonzero $y_{g} \in D_{g}$,  where,  if $h\in H$,  $y_h$ is chosen to commute with $D_e$.   Consider the element $z(x_{s_0}) = \sum_{g\in G}y_{g}x_{s_{0}}y_{g}^{-1}$. We claim $z(x_{s_0})\in K$: Because $x_{s_{0}}$ is in $L_{1}$ and conjugation by $y_{g}$ preserves $L_{1}$,  we see that $z(x_{s_0})$  lies in $L_1$ and in particular commutes with the elements of $D_e$. Therefore  it suffices  to show that for all $g\in G$, $z(x_{s_0})$ is fixed under conjugation by $y_{g}$. But this is clear because conjugation  by $y_{g}$ permutes the components modulo nonzero elements in $D_{e}$ and the elements of $D_e$ commute with $x_{s_{0}}$. So $z(x_{s_0})$ lies in $K$.  Note that if $r\in L^{*}$  then  $rx_{s_0}$ lies in $D_H$ and commutes with $D_e$,  so substituting we obtain $z(rx_{s_0})=\sum_{g\in G}y_{g}rx_{s_{0}}y_{g}^{-1}$ lies in $K$.  We {\it claim} we can find such an $r\in L^{*}$ such that  $z(rx_{s_0})$ is nonzero.  Because all of the homogeneous elements in the expression for $z(rx_{s_0})$ lie outside of $S_1$ it will follow that there are at least two nonzero terms in the sum and so we will have shown $K$ is nongraded, as desired.

To prove the claim, let
$C_{G}(s_{0})$ denote the centralizer of $s_{0}$ in $G$.   
For every $g \in C_{G}(s_{0})$ let  $g(x_{s_{0}})$ denote $y_{g}x_{s_{0}}y_{g}^{-1}$.   Then $g(x_{s_0}) = \beta_{g}x_{s_{0}}$ for some $\beta_g\in L^{*}$.  We let $\beta: C_{G}(s_{0}) \rightarrow L^{*}$ be given by $\beta(g)=\beta_g$. We  want to show  $\beta$ is a $1$-cocycle. Indeed, if $g_1,g_2\in C_G(s_0)$, then

$$g_{1}g_{2}(x_{s_{0}}) = g_{1}(g_{2}(x_{s_{0}}))$$

The left hand side yields
$$\beta(g_{1}g_{2})x_{s_{0}}$$

and  the right hand side yields
$$g_{1}(\beta(g_{2})x_{s_{0}}) = g_{1}(\beta(g_{2}))\beta(g_{1})x_{s_{0}}$$ so the result follows.

Now, because $s_0\in S$, we have $\beta(h)=1$ for all $h\in H$ and hence $\beta$  is induced by the inflation map from a $1$-cocycle (which we call $\bar{\beta}$) which represents  a class in $H^1(C_{G}(s_{0})/H,L^{*})$.  The group $C_{G}(s_{0})/H$ acts faithfully on $L$ and so, applying Hilbert's Theorem $90$ for $\bar{\beta}$, there is an element $t \in L^{*}$ such that  $\beta(g) = g(t)t^{-1}$ for every $g \in C_{G}(s_{0})$.
We can therefore write $g(x_{s_{0}}) = g(t)t^{-1}x_{s_{0}}$ which implies that for all $g \in C_{G}(s_{0})$,
$g(t^{-1}x_{s_0})= t^{-1}x_{s_0}$.
We now compute
$$z(t^{-1}x_{s_0})=\sum_{g\in G}y_{g}t^{-1}x_{s_{0}}y_{g}^{-1}=\sum_{g\in C_{G}(s_{0})}y_{g}t^{-1}x_{s_{0}}y_{g}^{-1}+\sum_{g\not\in C_{G}(s_{0})}y_{g}t^{-1}x_{s_{0}}y_{g}^{-1}=$$

$$\sum_{g\in C_{G}(s_{0})}g(t^{-1}x_{s_{0}})+\sum_{g\not\in C_{G}(s_{0})}y_{g}t^{-1}x_{s_{0}}y_{g}^{-1}=\sum_{g\in C_{G}(s_{0})}t^{-1}x_{s_{0}}+\sum_{g\not\in C_{G}(s_{0})}y_{g}t^{-1}x_{s_{0}}y_{g}^{-1}=
$$
$$t^{-1}|C_G(s_0)|x_{s_0}+\sum_{g\not\in C_{G}(s_{0})}y_{g}t^{-1}x_{s_{0}}y_{g}^{-1},
$$
which is nonzero because $|C_G(s_0)|$ is nonzero because the characteristic of the base field is zero! This finishes the proof of the claim.

We are left with determining the structure of $K$ when $K$ is graded, that is when $S$ is central in $G$.  In that case if $s\in S$,  then $C_G(s)=G$ and the argument in the preceding paragraph shows that there is a representative $x_s$ of $s$ such that $x_s$ lies in $K$.  It follows that $K=K_0[x_s : s\in S]=K_0^{\tilde\alpha}S$ where $\tilde\alpha:S\times S\rightarrow K^{*}_0$ is a two-cocycle and $\tilde\alpha$ is cohomologous to $\alpha$ in $H^2(S,L^{*})$, as desired.
\end{proof}

\section{construction of division algebras graded by $G$}

In this section we prove that given a group extension

$$\beta: 1 \rightarrow H \rightarrow G \rightarrow Q \rightarrow 1$$
where $H$ is abelian, a positive integer $d$, and a $Q$-invariant map $\phi: M(H) \rightarrow \mu_{n_{H}}$, $n_{H} = exp(H)$, there is a finite dimensional division algebra over its center, faithfully $G$-graded which realizes the given data. Here, the action of $Q$ on $M(H)$ is the induced action of $Q$ on $H$ and trivial on $\mu_{n_{H}}$.

We construct first division algebras $D$ in which $d$, the degree of $D_e$, equals $1$. That is, $L = D_e$ is a field. Thereafter we show how to pass from $d=1$ to an arbitrary $d$.

\begin{rem}\label{number of roots} In our construction the base field $k$ will contain $\zeta_{n}$, a primitive $n$th root of unity, where $n = l.c.m (d,exp(H))$. There  is no ``upper bound'' on the number of roots of unity in $k$; our construction works even if $k$ contains an algebraically closed field.
\end{rem}

Let us start our discussion with two special cases: in case $(1)$ we assume $H = \{1\}$, and in case $(2)$ we assume $G = H$.
The case where $H$ is trivial is well known. We recall here a proof which is attributed to S. Rosset and K. Brown (see \cite{Brown} and \cite{Rosset}). Their idea will appear in the proof of the general case.

Let
$$
1 \rightarrow R \rightarrow F \rightarrow G \rightarrow 1
$$
be a presentation of the finite group $G$ where the groups $F$ and $R$ are finitely generated free.

Taking quotient groups modulo the commutator subgroup $[R,R]$ we obtain
$$
\beta : 1 \rightarrow N = R/[R,R] \rightarrow \Gamma = F/[R,R] \rightarrow G \rightarrow 1.
$$

We want the action of $G$ on $N$ in this extension to be faithful,  that is,  we want the centralizer of $N$ in $\Gamma$ to be $N$ itself.  This can be arranged as follows:  Let $A$ be the free abelian group on $d_{G} (=|G|)$ generators and let $G$ act on $A$ by left translation on these generators, that is, the regular permutation representation of $G$.  We obtain the split exact sequence $1\rightarrow A\rightarrow A\rtimes G\rightarrow G\rightarrow 1$.  We take our group $F$ to be free on $2d_{G}$ generators $y_g,x_g$ for all $g\in G$ and map $F$ to $G$ by sending each  $y_g$ to the identity and each $x_g$ to $g$.  The group $F$ maps to $A\rtimes G$ by sending the $y_g's$ to generators of $A$ and each $x_g$ to $g$.  The resulting extension denoted above by $\beta$ maps to this split exact sequence and a simple diagram chase then shows  that $G$ acts faithfully on $N$.

By a theorem of Higman (\cite{Hig}, Theorem $2$), $\Gamma=F/[R,R]$ is torsion free.  Because $\Gamma$ is abelian by finite, it follows from a theorem of Brown that the group algebra $k\Gamma$ is a domain (see \cite{Brown}, Cor. $2$). Let $S$ denote the ring $kN$, a commutative subring of $k\Gamma$ and hence an integral domain.  The group $G$ acts faithfully on $S$ and the  crossed-product ring $S * G$ is a domain because it is isomorphic to $k\Gamma$. Let $R$ denote the fixed ring of the action of $G$ on $S$.  Then $R$ is the center of $S*G$.  Because the ring $R$ is the fixed ring of a finite group action,  the extension $S/R$ is integral. It follows that if $K$ denotes the field of fractons of $R$ and $L$ the field of fractions of $S$ then $L=SK$.  The extenion $L/K$ is a Galois extension with Galois group $G$.  The classical crossed-product algebra $L*G$ is a division algebra:  $L*G\cong (S*G)\otimes_RK$ and every element of $L*G$ can be written as $r^{-1}a$ where $r\in R$ is nonzero and $a\in S*G$.  It follows that $L*G$ is a domain and finite dimensional over $K$, so a division algebra. Morevoer $L*G$  realizes the given data for this case.

We next consider case $(2)$ in which $G = H$ is abelian and $\phi$ is any map $\phi: M(H)\rightarrow \mu_{n_{H}} \subset k^{*}$, $n_{H} = exp(H)$. Note that in this case every cohomology class is invariant because the inner action of $H$ acts trivially on $M(H)$.

Let
$$
1 \rightarrow R \rightarrow F_{H} \rightarrow H \rightarrow 1
$$
be a presentation of $H$.
We take quotients modulo $[R,F_{H}]$ and obtain the central extension
$$
1 \rightarrow R/[R,F_{H}] \rightarrow F_{H}/[R,F_{H}] \rightarrow H \rightarrow 1.
$$

The group $R/[R,F_{H}]$ is finitely generated abelian and so  is the direct product of a finite group $U$ and a finitely generated torsion free group $T\cong \mathbb{Z}^{r}$. We \underline{claim} (and it is well known) that the torsion part is precisely $M(H)$, the Schur multiplier of $H$.
Indeed, by the Hopf formula the Schur multiplier is given by $(R\cap [F_{H},F_{H}])/[R,F_{H}]$ so it is necessary and sufficient to show that
$$(R/[R,F_{H}])/ ((R\cap [F_{H},F_{H}])/[R,F_{H}]) \cong R/ (R\cap [F_{H},F_{H}])$$
is torsion free. But $R/ (R\cap [F_{H},F_{H}])$ is isomorphic to $R[F_{H},F_{H}]/[F_{H},F_{H}]$, a subgroup of the torsion free  group $F_{H}/[F_{H},F_{H}]$. This proves the claim.



Let $\Gamma_H=F_{H}/[R,F_{H}]$. The group $U$ is characteristic in $R/[R,F_{H}]$ and so normal in $\Gamma_{H}$. We therefore have the extension
$$\beta: 1 \rightarrow U \rightarrow \Gamma_{H} \rightarrow \Gamma_{H}/U \rightarrow 1$$

which is easily seen to be the inflation of the extension
$$\bar{\beta}: 1 \rightarrow U \rightarrow \Gamma_{H}/T \rightarrow H\rightarrow 1.$$

We \underline{claim} $\Gamma_{H}/U$ is torsion free:
Indeed, $$\Gamma_{H}/U = (F_{H}/[R,F_{H}])/((R \cap [F_{H}, F_{H}])/[R, F_{H}])\cong$$
$$ F_{H}/(R \cap [F_{H}, F_{H}]) = F_{H}/[F_{H}, F_{H}]$$ because $H$ is abelian.  The claim follows.

The extension $\beta$ yields a crossed product $(kU) ^{\beta}\Gamma_{H}/U$. Composing the map $\phi$, whose image is in $\mu_{n_{H}}$, with the $2$-cocycle $\beta$ we obtain a cocycle which represents a class in $H^{2}(\Gamma_{H}/U, \mu_{n_{H}})$ and which we denote by $\phi(\beta)$. Thus we obtain the twisted group algebra $k^{\phi(\beta)}\Gamma_{H}/U$. Because $\Gamma_{H}/U$ is torsion free and abelian by finite, a theorem of J. Moody (see for instance \cite{Passman}, Lemma, $37.8$) implies that this twisted group algebra   is an Ore domain.
Because $\beta$ is the inflation of $\bar{\beta}$,  the elements of $T$ in   $k^{\phi(\beta)}\Gamma_{H}/U$ are central and hence, by inverting the elements of $kT\setminus \{0\}$, we obtain the twisted group division algebra  $k(T)^{\phi(\bar{\beta})}H$ over the field $k(T)$.  Letting  $k(T)^{\phi(\bar{\beta})}H=\sum_{h\in H}k(T)x_h$ we compute easily that for all  $h_1,h_2\in H$,  the commutator $x_{h_1}x_{h_2}x_{h_1}^{-1}x_{h_2}^{-1}=\phi(h_1\wedge h_2)$ and so, by the comments preceding Definition \ref{realizable triple definition}, this division algebra realizes the given data.

We now turn to the proof of the general case.
Let
$$
1 \rightarrow R \rightarrow F_{G} \rightarrow G \rightarrow 1
$$
be a presentation of the finite group $G$ (We will be more precise concerning the choice of $F_G$ below).
We denote the map $F_{G} \rightarrow G$ by $\pi$.

For the subgroup $H$ we obtain the following induced extension,  where  $F_{H} = \pi^{-1}(H)$:
$$
1 \rightarrow R \rightarrow F_{H} \rightarrow H \rightarrow 1
$$

Taking quotients modulo $[R,F_{H}]$ in these two  extensions we obtain
$$
1 \rightarrow R/[R,F_{H}] \rightarrow F_{G}/[R,F_{H}] \rightarrow G \rightarrow 1
$$

\noindent and
$$1 \rightarrow R/[R,F_{H}] \rightarrow F_{H}/[R,F_{H}] \rightarrow H \rightarrow 1.
$$

As in case (2),  the  group $R/[R,F_{H}]$  is the direct product of a finite group $U=M(H)$ and a finitely generated torsion free group $T\cong \mathbb{Z}^{r}$.  Moreover $U$ is characteristic in $R/[R,F_{H}]$ and hence normal in both $\Gamma_H=F_{H}/[R,F_{H}]$ and $\Gamma_G=F_{G}/[R,F_{H}] $.  The group $\Gamma_G/U$ is torsion free: We have $\Gamma_G/U\cong (F_G/[R,F_H])/((R\cap[F_H,F_H])/[R,F_H])\cong F_G/[F_H,F_H]$ because $H$ is abelian, and $F_G/[F_H,F_H]$ is torsion free by Higman's theorem.

We have the extension:
$$
\beta: 1 \rightarrow U \rightarrow \Gamma_{G} \rightarrow \Gamma_{G}/U \rightarrow 1.
$$

Because $\phi: M(H)\rightarrow \mu_{n_{H}}$ is $G$ invariant, the map  $\phi\circ \beta:\Gamma_G/U
\times \Gamma_G/U\rightarrow \mu_{n_{H}}$ is a $2$-cocycle.  Therefore we can form  the
twisted group algebra
$k ^{\phi\circ \beta}\Gamma_{G}/U$.
The group $\Gamma_{G}/U$ is abelian by finite and we have seen that it is torsion free. Again as in case (2) we infer that  $k ^{\phi\circ \beta}\Gamma_{G}/U$ is an Ore domain.

We claim the ring of quotients of $k ^{\phi\circ \beta}\Gamma_{G}/U$ is a division algebra, $G$-graded, satisfying the desired conditions.

What do we have to prove? Because it is an Ore domain the ring of quotients is a division algebra which we denote by $D$. So we need to prove that $D$ is finite dimensional over its center, $G$-graded and the grading gives rise to the given group extension
$$
1 \rightarrow H \rightarrow G \rightarrow Q \rightarrow 1
$$
where $H$ is abelian, and the given $Q$-invariant map $\phi: M(H) \rightarrow \mu_{n_{H}}$.

We first consider the subalgebra  $k ^{\phi\circ \beta}\Gamma_{H}/U$.  This is the same as the dvision algebra obtained in case (2) and as we saw there the cocycle $\beta$ restricted to $\Gamma_{H}/U$ is inflated from the extension
$$ \bar{\beta}: 1 \rightarrow U \rightarrow \Gamma_{H}/T \rightarrow H\rightarrow 1.$$

Therefore the elements
of $T$ in $k^{\phi(\beta)}\Gamma_{H}/U$ are central and hence, by inverting the elements of $kT\setminus \{0\}$, we obtain the twisted group division algebra  $k(T)^{\phi(\bar{\beta})}H$ over the field $k(T)$.  Also, as in case (2), the cocycle on $H$ maps to $\phi$ as desired.

From the extension $1\rightarrow T\rightarrow \Gamma_{G}/U\rightarrow G\rightarrow 1$  we see that if we set $D_e=k(T)$ and choose representatives $x_g$ in $\Gamma_{G}/U$ for the elements $g$ in $G$, then $D=\sum_{g\in G}D_ex_g$ is faithfully $G$-graded. We are then left with showing that $H$ is exactly the kernel of the action of $G$ on $k(T)$.  For that we need to be more precise about our choice of $F_G$:   What is required is that in the sequence
$1 \rightarrow T \rightarrow F_{G}/U\rightarrow G \rightarrow 1
$ the kernel of the action of $G$ on $T$ is exactly $H$ (We already know the kernel contains H).    We proceed as in case (1):  Let $A$ be the free abelian group on $d_{Q}=|Q|$ generators and let $G$ act on $A$ by left translation on these generators.   We obtain the split exact sequence $1\rightarrow A\rightarrow A\rtimes G\rightarrow G\rightarrow 1$.  We take our group $F_G$ to be free on $d_{Q} + d_{G}$ generators $y_{g_1}, y_{g_2},  \dots, y_{g_{d_{Q}}}$,  where $g_1,g_2, \dots , g_{d_{Q}}$ are distinct coset representatives of $H$ in $G$ together with generators $x_g$ for all $g\in G$.  We map $F_G$ to $G$ by sending each  $y_{g_i}$ to the identity and each $x_g$ to $g$.  We map $F_G$  to $A\rtimes G$ by sending the ${y_{g_i}}'s$ to generators of $A$ and each $x_g$ to $g$.

With this choice of presentation for $G$ a straightforward calculation shows that we have an induced map of extensions from
$1 \rightarrow R/[R,F_{H}] \rightarrow F_{G}/[R,F_{H}] \rightarrow G \rightarrow 1
$ to $1\rightarrow A\rightarrow A\rtimes G\rightarrow G\rightarrow 1$, where the map on $G$ is the identity. Because $U$ is torsion and  $A$ is torsion free we get an induced map of extensions from $1 \rightarrow T \rightarrow \Gamma_{G}/U\rightarrow G \rightarrow 1 $ to $1\rightarrow A\rightarrow A\rtimes G\rightarrow G\rightarrow 1$. A simple diagram chase now shows that the kernel of the action of $G$ on $T$ is exactly $H$, as desired.

We conclude this section by showing how to pass from $d=1$ to an arbitrary $d$. To this end, we shall assume (as we may by Remark \ref{number of roots}) the field $k$ in the construction (and hence also $K_{0}$ and $L$) contains $\zeta_{d}$, a primitive $d$th root of unity.

Adjoin new indeterminates $a,b$ to $L$ and let
$$E=(a,b)_{L(a,b),d} = <x, y: x^{d} = a, y^{d} = b, yx = \zeta_{d}xy>$$
be the symbol algebra of degree $d$ over $L(a,b)$ determined by $a$ and $b$. Because $a$ and $b$ are indeterminates the algebra $E$ is a division algebra of degree $d$. Next, we are assuming we have constructed a division algebra $D$ realizing the extension $\beta$, $\phi$ and $d=1$. This means that $D$ is isomorphic to a crossed product of $L^{\alpha}H \ast Q$. We denote by $K_{0}$ its $e$-center. Now, it is well known that a rational extension $E(x)$ of $E$, $E$ any field, induces an injective map $res_{E(x)}^{E}: Br(E)\rightarrow Br(E(x))$ of the Brauer groups (see \cite{Fadeeev}, \cite{AusBrumer}) and hence extending the $e$-center of $L^{\alpha}H \ast Q$ to $K_{0}(a,b)$ yields a division algebra $L(a,b)^{\alpha}H \ast Q$ which realizes the same data as $D$. Finally, the algebra we are looking for is
$$
\tilde{D} \cong ((a,b)_{L(a, b), d}\otimes L^{\alpha}H)\ast Q
$$
where the outer action of $Q$ on $(a,b)_{L( a, b),d}\otimes L^{\alpha}H$ is determined via the outer action of $Q$ on $L^{\alpha}H$. It is clear the constructed algebra is $G$-graded and realizes the triple with the same $\beta$ and $\phi$ as $D$, but with $d$ rather than $1$. We need to show this is a division algebra. To this end it is convenient to view the algebra $\tilde{D}$ as a Hilbert twist of $D$. Indeed, extend $K_{0}$, the $e$-center of $D$, by the central indeterminate $a^{1/d}$ and denote it by $D_{K_{0}(a^{1/d})}$. This is a division algebra by the fact above and because a transcendental extension of the center yields an injective map
$$
res^{L}_{L(a^{1/d})}: Br(L) \rightarrow Br(L(a^{1/d})).
$$
Consider the Hilbert twist extension $D_{K_{0}(a^{1/d})}(\sigma,y)$. This is the ring of quotients of the twisted polynomial ring $D_{K_{0}(a^{1/d})}[\sigma,y]$ where $\sigma$ is the automorphism of $D_{K_{0}(a^{1/d})}$ mapping $a^{1/d}$ to $\zeta_{d}a^{1/d}$ and $y$ skew-commutes with elements of $D_{K_{0}(a^{1/d})}$ via the rule $zy = y\sigma(z)$. The algebra obtained is $G$-graded isomorphic to $\tilde{D}$ where $b = y^{d}$. We conclude that because $D_{K_{0}(a^{1/d})}$ is a division algebra, the algebra $D_{K_{0}(a^{1/d})}(\sigma,y)$ is a division algebra as desired (see \cite{Lam}, Example $1.8$). This completes the proof of the first part Theorem \ref{Main realizable triple theorem}. The second part follows easily; if the cocycle is nondegenerate, the center of $L^{\alpha}H$ is $L$ which is contained in $D_{e}$.

\section{Forms of finite dimensional $G$-graded simple algebras}

In this section we prove Theorem \ref{Simple algebras forms}.

\begin{proof} Let $A$ be a finite dimensional $G$-graded simple algebra over its $e$-center $F$, where $F$ is algebraically closed of characteristic zero. Let $P_{A} = \{H, (e, g_{2},\ldots, g_{s}), \alpha \}$ be a Bahturin, Sehgal, Zaicev presentation. Suppose first it admits a division algebra $G$-graded form $D$ over $K_{0}$ (the $e$-center of $D$). because $D$ is in particular a $G$-graded division algebra (that is, nonzero homogeneous elements are invertible) the group $H$ which appears in the presentation must be normal, all coset representatives appear and with equal frequency which we denote by $d$. Moreover, the class $[\alpha]$ is $G$-invariant (see \cite{AljadeffKarasikVerbally}, Corollary $1.13$).
To complete the proof of the necessity it remains to show $H$ is abelian. Now, because the $e$-component of $D$ extends to the $e$-component of $A$, and $D_{N}$ extends to $A_{N}$ for any subgroup $N$ of $G$, the result will follow if we show the elements in $D$ that centralize $Z(D_{e})$ belong to the $N$-homogeneous component where on the one hand $N$ is abelian and on the other hand, these are precisely the elements of $A_{H}$. The first statement follows from the structure theorem. For the proof of the second statement note that the center of $A_{e}$ is spanned over $F$ by the $r$ elements (idempotents)
$$u_{e}\otimes (e_{1,1} +\ldots+e_{d,d}), \ldots, u_{e}\otimes (e_{(r-1)d+1,(r-1)d+1} +\ldots+e_{rd,rd})$$
where $r = [G:H]$. because $H$ is normal in $G$, $A_{H}$ is spanned by elements of the form $u_{h}\otimes_{F}e_{i,i}$ which centralizes $Z(A_{e})$. Let us show these are all. Let
$$
w = \sum_{h\in H, i,j}\gamma_{(h, i,j)}u_{h}\otimes e_{i,j}, \gamma_{(h, i,j)}\in F
$$ be an element in $A$ that centralizes $Z(A_{e})$. Multiplying from left and right by idempotents $u_{e}\otimes (e_{(k-1)d+1,(k-1)d+1}+\ldots+e_{kd,kd})$, $k=1,\ldots,r$, we see $\gamma_{(h, i,j)} = 0$ unless $i$ and $j$ belong to the same block, that is $i,j \in \{(k-1)d+1,\ldots,kd\}$, some $k=1,\ldots,r$.

Conversely: We want to show that any such algebra $A$, with presentation $$P_{A} = \{H, (e, g_{2},\ldots, g_{s}), \alpha \}$$ has a division algebra form. First we say that it is sufficient to consider the case where $d = 1$. Indeed, if $P_{A} = \{H, (e, g_{2},\ldots, g_{s}), \alpha \}$, where each coset representative appears $d$ times, that is $dr = s$, we consider the corresponding algebra $\bar{A}$ with $P_{\bar{A}} = \{H, (e, \nu_{2},\ldots,\nu_{r}), \alpha\}$ (i.e. the same presentation but with $d=1$). If $\bar{D}$ is a division algebra form for $\bar{A}$, where its $e$-center contains $\zeta_{d}$, arguing as in last part of the previous section we obtain that the corresponding Hilbert twist is a division algebra $G$-graded form of $A$.

Changing notation we consider a $G$-graded simple with presentation $$P_{A} = \{H, (e, g_{2},\ldots, g_{s}), \alpha \},$$ where $H$ is abelian and normal, all coset representatives appear exactly once and $[\alpha]$ is $G$-invariant. We need to construct a division algebra form. Consider the triple $(\beta, \pi(\alpha), 1)$ induced by $P_{A}$, where
$\beta: 1\rightarrow H \rightarrow G \rightarrow G/H \rightarrow 1$ and $\pi(\alpha): Hom(M(H)\rightarrow \mu_{n_{H}}$, $n_{H} = exp(H)$, and let $D$ be a division algebra which realizes $(\beta, \pi(\alpha), 1)$. In particular we have $D_{e} = L$, $D_{H} \cong L^{\alpha{'}}H$, $K_{0} = L^{G/H}$ and $H$ is the centralizer of $L$ in $G$. We want to show $D$ is a $G$-graded form of $A$. Suppose extension of scalars of $D$ yields a $G$-graded simple algebra $A'$ with presentation $P_{A'} = \{H', (e, g'_{2},\ldots, g'_{s}), \alpha' \}$. because the $e$-component is commutative and its dimension is on the one hand $[G:H]$ (extension of the form) and on the other hand, from the presentation $P_{A}$, it is $[G:H']$, we have that $H$ and $H'$ have the same order and each coset representative in $P_{A'}$ must appear once. Furthermore because the homogeneous elements in $D$ which centralize the $e$-component are precisely the $H$-homogeneous elements, this carry out for $A'$ and this forces $H = H'$. To complete the proof we need to show that extension of scalars yields the cohomology class in $H^{2}(H,F^{*})$ (in the presentation $P_{A'}$) which corresponds by means of the Universal Coefficient Theorem to the map $\pi([\alpha]): M(H) \rightarrow \mu_{n_{H}}$. This follows from the fact that the $H$ identities in $D_{H}$ are determined by the map $\pi([\alpha]): M(H) \rightarrow \mu_{n_{H}}$ and the $H$-identities of $A'_{H}$ are determined by the map $\pi([\alpha']): M(H) \rightarrow \mu_{n_{H}}$ together with the fact that upon a ($char = 0$) field extension of the $e$-center the $H$-graded identities do not change.
\end{proof}
\begin{acknowledgment}
We thank the referee for his/her very valuable report.
\end{acknowledgment}

\end{document}